\documentclass[12]{amsart}
\usepackage{amsmath,amssymb,amsthm,color,enumerate,comment,centernot,enumitem,url,cite}
\usepackage{graphicx,relsize}
\usepackage{mathtools}
\usepackage{array}
\newcommand{\ord}{\operatorname{ord}}

\newtheorem{thm}{Theorem}[section]
\newtheorem{lem}[thm]{Lemma}

\newtheorem{prop}[thm]{Proposition}
\newtheorem{cor}[thm]{Corollary}

\theoremstyle{remark}

\theoremstyle{definition}
    \newtheorem{defn}[thm]{Definition}

\newtheorem{rem}[thm]{Remark}

\theoremstyle{THM}

\DeclareMathOperator{\lcm}{lcm}

\newcommand{\abs}[1]{\left|{#1}\right|}

\def\P {{\mathcal P}}
\def\O {{\mathcal O}}

\def\Z {{\mathbb Z}}

\def\NN {{\mathcal N}}

\def\Q {{\mathbb Q}}

\def\P {{\mathcal P}}
\def\F {{\mathbb F}}

\def\Z {{\mathbb Z}}
\def\Q {{\mathbb Q}}

\def\Gal{{\mbox {Gal} }}

\newcommand{\OK}{\mathcal{O}}
\newcommand{\ddiv}{\, \mid \,}
\newcommand{\ndiv}{\, \nmid \,}
\newcommand{\Mod}[1]{\ (\mathrm{mod}\ #1)}

\def\red#1 {\textcolor{red}{#1 }}
\def\blue#1 {\textcolor{blue}{#1 }}

\numberwithin{equation}{section}

\def\Z {{\mathbb Z}}

\begin{document}

\title[Monogenic trinomials with non-squarefree discriminant]{Monogenic trinomials with non-squarefree discriminant}
\author{Lenny Jones}
\address{Shippensburg University, Department of Mathematics, 1871 Old Main Drive, Shippensburg, PA 17257}
\email{lkjone@ship.edu}
\author{Daniel White}
\address{Bryn Mawr College, Department of Mathematics, 101 North Merion Avenue, Bryn Mawr, PA 19010}
\email{dfwhite@brynmawr.edu}

\begin{abstract}
For each integer $n\ge 2$, we identify new infinite families of monogenic trinomials $f(x)=x^n+Ax^m+B$ with non-squarefree discriminant, many of which have small Galois group. Moreover, in certain situations when $A=B\ge 2$ with fixed $n$ and $m$, we produce asymptotics on the number of such trinomials with $A\le X$.
\end{abstract}


\maketitle

\section{Introduction}

Unless stated otherwise, polynomials $f(x)\in \Z[x]$ are assumed to be monic, and when we say $f(x)$ is ``irreducible", we mean irreducible over $\Q$.
Let $K$ be an algebraic number field of degree $n$ over $\Q$. For any $\theta\in K$, we let
\[\Delta(\theta):=\Delta\left(1,\theta,\theta^2,\ldots,\theta^{n-1}\right)\] denote the discriminant of $\theta$. Similarly, we let $\Delta(f)$ and $\Delta(K)$ denote the discriminants over $\Q$, respectively, of the polynomial $f(x)$ and the field $K$. If $f(x)$ is irreducible, with $f(\theta)=0$ and $K=\Q(\theta)$, then we have the well-known equation \cite{Cohen}
\begin{equation}
\label{Eq:Dis-Dis}
\Delta(f)=\Delta(\theta)=\left[\OK_K:\Z[\theta]\right]^2\Delta(K),
\end{equation}
where $\OK_K$ is the ring of integers of $K$.
Recall that $K$ is called \emph{monogenic} if there exists some element $\theta\in \OK_K$ such that $\OK_K=\Z[\theta]$. In other words,
$\left\{1,\theta, \theta^2,\ldots ,\theta^{n-1} \right\}$ is a \emph{power integral basis} for $\OK_K$. One advantage of the monogenic situation is that arithmetic calculations in $\OK_K$ are then much easier. One could argue that quadratic and cyclotomic fields are, in part, better understood because they are monogenic. For more information on general monogenic fields, see \cite{N}. For some recent specific examinations of monogenic fields, see
\cite{ANHam,ANHus,Gassert,BMT,ESW,Gaal,GR2,GR1,JK,S1}.

We see from \eqref{Eq:Dis-Dis} that $K$  being monogenic is equivalent to the existence of some irreducible polynomial $f(x)$, with $f(\theta)=0$ and $K=\Q(\theta)$, such that $\Delta(f)=\Delta(K)$. Therefore, given an irreducible polynomial $f(x)$ with $f(\theta)=0$, we can ask when the field $K=\Q\left(\theta\right)$ is monogenic. Note that it is possible to have $\OK_K\ne \Z\left[\theta\right]$ for such a polynomial $f(x)$ even though $K$ is monogenic. This phenomenon occurs frequently for quadratic polynomials. For a nontrivial example \cite{AW}, suppose that $f(\theta)=\theta^3-3\theta+9=0$. Then
\[\Delta(f)=\Delta(\theta)=-3^3\cdot 7\cdot 11 \ne -3\cdot 7 \cdot 11=\Delta(K),\] and hence $\left\{1,\theta,\theta^2\right\}$ is not an integral basis for $\OK_K$ since $\left[\OK_K:\Z[\theta]\right]=3$ by \eqref{Eq:Dis-Dis}. Note that $f\left(\theta^2/3\right)\ne 0$. However, it is easy to check that
\[
\left(\theta^2/3\right)^3-2\left(\theta^2/3\right)^2+\left(\theta^2/3\right)-3=0,
\]
so that $\theta^2/3\in \OK_K$. Since
\[
\Delta(\theta^2/3)=-3\cdot 7 \cdot 11=\Delta(K),
\]
it follows that $\left\{1,\theta^2/3,\theta^4/9\right\}$ is a power integral basis for $\OK_K$, and $K=\Q(\theta)$ is indeed monogenic. We are then motivated to define a \emph{monogenic polynomial} to be an irreducible polynomial $f(x)$ such that   $\OK_K=\Z\left[\theta\right]$, where $f\left(\theta\right)=0$. In this case, we also refer to the field $K=\Q(\theta)$ as the \emph{monogenic field} of $f(x)$.

By \eqref{Eq:Dis-Dis}, we see that a strategy to guarantee that $f(x)$ is monogenic is to determine conditions for which $\Delta(f)$ is squarefree. This is precisely the procedure used in \cite{BMT,K}. In \cite{BMT}, the focus is on trinomials of the form $f(x)=x^n+ax+b$, where it is proven that $\delta \le .99344674$ and  conjectured that $\delta \ge .9934466$, where $\delta$ is the density of the set of positive integers $n$ such that the discriminant of the polynomial $x^n-x-1$ is squarefree (another analysis of $\delta$ can be found in \cite{Shpar}). In \cite{K}, a more general construction of polynomials with squarefree discriminant is given without regard to a specific form of the polynomials. This construction is achieved using methods found in \cite{Helfgott}, and it establishes the existence of infinitely many monic polynomials of a given degree with squarefree discriminant.

In \cite{BSW} it is proven that, as the degree tends to infinity, the probability that a random monic polynomial $f(x)\in \Z[x]$ has squarefree discriminant approaches approximately $0.358232$. Certainly, a squarefree discriminant is not necessary for $f(x)$ to be monogenic, and indeed, it is shown in \cite{BSW} that the density of the monic monogenic polynomials for any fixed degree $\ge 2$ (as the coefficents grow in a prescribed way) is $1/\zeta(2) \approx 0.607927$.

More recently in \cite{LJPAMS,LJActa}, new infinite families of monogenic polynomials with non-squarefree discriminant were given. A construction similar to the one used in \cite{K} was used in \cite{LJPAMS} for prime-degree polynomials, while a completely different technique  was used in \cite{LJActa} for polynomials of arbitrary degree. In general, the polynomials produced in these articles are not trinomials.
\subsection{Main results}
 Jakhar,  Khanduja and Sangwan \cite{JKS2, JKS1} have given necessary and sufficient conditions, based solely on $n,m,A,B$, for an irreducible trinomial $f(x)=x^n+Ax^m+B$ to be monogenic. While examples are provided in \cite{JKS1} for the situation where $n=tm+u$, with $0\le u\le m-1$, and either $u=0$ or $u$ is a divisor of $m$, no indication is given as to whether there exist infinite families of such trinomials. In this article we use the conditions given in \cite{JKS1}, together with some asymptotic results (see Theorem \ref{mainasymp} and Theorem \ref{Thm:Pasten}), to construct infinite families of monogenic trinomials  with non-squarefree discriminant, where $n\ge 2$ is an integer and $m\ge 1$ is a proper divisor of $n$. We point out that the elements of the families given here are of a different form than previously examined in the literature. Our results are as follows.

\begin{thm}
\label{first_type}
Let $n\ge 2$ be an integer, $m\ge 1$ be a proper divisor of $n$, and $t=n/m$. Let $\kappa$ denote the squarefree kernel of $m$. Then there exist exactly
\begin{equation}
\label{first_asymp}
\frac{X}{\kappa \zeta(2)} \prod_{p \ddiv \kappa} \left( 1 - \frac{1}{p + 1} \right) \prod_{p \ndiv t(t-1)\kappa} \left( 1 - \frac{1}{p^2-1} \right) + O(X^{3/4})
\end{equation}
positive integers $A \le X$ such that $A \equiv 0 \pmod{\kappa}$ and $f(x)=x^{n}+Ax^{m}+A$ is monogenic with non-squarefree discriminant. 
The implied constant in \eqref{first_asymp} is dependent on $n$ and $m$.
\end{thm}

It is curious that, for many pairs $(n,m)$, the quantity in \eqref{first_asymp} is asymptotic to $cX$ where $c$ is very near (but never exceeds) $1/ (\kappa \zeta(2))$; see comments on \cite{BSW} above. This phenomenon occurs when $t(t-1)$ is divisible by all primes less than a ``large" magnitude and $m$ has a ``reasonable" number of factors not consisting of ``small" primes. In any case, the situation $m=1$ is particularly interesting, as the divisibility condition on $A$ becomes trivial, yielding a complete classification of monogenic polynomials of the form $x^n + Ax + A$ (with $A \ge 2$), essentially all of which have non-squarefree discriminant. 

\begin{thm}
\label{second_type}
Let $n\ge 2$ be an integer, $m\ge 1$ be a proper divisor of $n$, and $t=n/m$. Let $\kappa$ denote the squarefree kernel of $m$ and suppose that $\kappa$ is a divisor of {\rm 210}. Then there exist exactly
\begin{equation}
\label{second_asymp}
\frac{X}{\kappa^2 \zeta(2)} \prod_{p \ddiv \kappa} \left( 1 - p^{-2} \right)^{-1} \prod_{p \ndiv t(t-1)\kappa} \left( 1 - \frac{1}{p^2 - 1} \right) + O(X^{3/4})
\end{equation}
positive integers $A \le X$ such that $A \equiv -1 \pmod{\kappa^2}$ and $f(x)=x^{n}+Ax^{m}+A$ is monogenic with non-squarefree discriminant. 
The implied constant in \eqref{second_asymp} is dependent on $n$ and $m$.
\end{thm}

Note that, for a fixed integer $n\ge 2$, the only overlap with the families of trinomials in Theorem \ref{first_type} and Theorem \ref{second_type} is when $m=1$. For the case $\kappa = 2$, Theorems \ref{first_type} and \ref{second_type} address all residue classes $A\not \equiv 1 \pmod{4}$, while the situation when $A\equiv 1 \pmod{4}$ is handled separately (see Proposition \ref{kappa2}). 
These results allow us to obtain a complete classification of monogenic polynomials of the form $x^{2^kt} + Ax^{2^k} + A$ (with $A \ge 2$, $t \ge 2$ and $k \ge 1$), all of which have non-squarefree discriminant.

The next theorem addresses, to some extent, the more general trinomials $f(x)=x^n+Ax^m+B$, where $A\ne B$.
\begin{thm}
\label{general_type}
Let $n\ge 2$ be an integer, $m\ge 1$ be a proper divisor of $n$, and $t=n/m$. Let $\kappa$ denote the squarefree kernel of $m$. Suppose $r$ is a prime such that $\kappa\not \equiv 0 \pmod{r}$, $A\equiv 0\pmod{r\kappa}$ and $\gcd(A/r\kappa,t)=1$. Then there exist infinitely many positive integers $B$ such that  $f(x)=x^n+A x^m+B$ is monogenic. This result is unconditional when $2\le t\le 4$ and conditional on the $abc$-conjecture for number fields when $t\ge 5$. For $m \ge 2$, infinitely many such $B$ exist where $\Delta(f)$ is not squarefree.
\end{thm}

Theorems \ref{first_type}, \ref{second_type}, and \ref{general_type} differ from many previous examinations of specific trinomial forms in the literature in that the discriminants of the trinomials here are not squarefree, and their Galois groups can be relatively small (see Proposition \ref{galois_prop} and the subsequent remark). Recently in \cite{JP}, families of monogenic trinomials have been examined where the discriminant is not squarefree. However, for the trinomials in \cite{JP}, the Galois group is either the symmetric group $S_n$ or the alternating group $A_n$, where $n$ is the degree of the trinomial.

\begin{prop}
\label{galois_prop}
Let $n\ge 2$ be an integer, $m\ge 1$ be a proper divisor of $n$ and $t=n/m$. If $f(x)=x^n+Ax^m+B$ is irreducible, then
\begin{equation}
\label{galois_bound}
|\Gal_{\Q}(f)| \le \varphi(m) m^t t!,
\end{equation}
where $\Gal_{\Q}\left(f\right)$ is the Galois group over $\Q$ of $f(x)$ and $\varphi$ is the totient function.
\end{prop}

For $n$, $m$ and $t$ as appearing in Theorems \ref{first_type}, \ref{second_type} and \ref{general_type}, note that the ratio of $\varphi(m) m^t t!$ to $|S_n|$ can be made arbitrarily small by taking $n$ large and $m$ commensurate with $n$.

\section{Basic Preliminaries}
\label{prelims}

We present some known facts that are used to establish the results in this article. We first state two theorems which allow us to study various properties relating to the monogeneity of a given monic trinomial using calculations involving only its coefficients and exponents.
\begin{thm}
\label{Thm:polydisc}
{\rm \cite{Swan}}
Let $f(x)=x^n+Ax^m+B\in \Z[x]$, where $0<m<n$, and let $d=\gcd(n,m)$. Then
\[
\Delta(f)=(-1)^{n(n-1)/2}B^{m-1}\left(n^{n/d}B^{(n-m)/d}-(-1)^{n/d}(n-m)^{(n-m)/d}m^{m/d}A^{n/d}\right)^d.
\]
\end{thm}


\begin{thm}{\rm \cite{JKS1}}\label{Thm:JKS1Main}
Let $n\ge 2$ be an integer.
Let $K=\Q(\theta)$ be an algebraic number field with $\theta\in \OK_K$, the ring of integers of $K$, having minimal polynomial $f(x)=x^{n}+Ax^m+B$ over $\Q$, where $m\ge 1$ is a proper divisor of $n$. Let $t=n/m$. A prime factor $p$ of $\Delta(f)$ does not divide $\left[\OK_K:\Z[\theta]\right]$ if and only if all of the following statements are true: 
\begin{enumerate}[font=\normalfont]
  \item \label{JKS:C1Main} if $p\mid A$ and $p\mid B$, then $p^2\nmid B$;
  \item \label{JKS:C2Main} if $p\mid A$ and $p\nmid B$, then
  \[\mbox{either } \quad p\mid a_2 \mbox{ and } p\nmid b_1 \quad \mbox{ or } \quad p\nmid a_2\left(a_2^tB+\left(-b_1\right)^t\right),\]
  where $a_2=A/p$ and $b_1=\frac{B+(-B)^{p^j}}{p}$ with $p^j\mid\mid tm$;
  \item \label{JKS:C3Main} if $p\nmid A$ and $p\mid B$, then
  \[\mbox{either } \quad p\mid a_1 \mbox{ and } p\nmid b_2 \quad \mbox{ or } \quad p\nmid a_1b_2^{m-1}\left(Aa_1^{t-1}+\left(-b_2\right)^{t-1}\right),\]
  where $a_1=\frac{A+(-A)^{p^\ell}}{p}$ with $p^\ell\mid\mid (t-1)m$, and $b_2=B/p$;
  \item \label{JKS:C4Main} if $p\nmid AB$ and $p\mid m$ with $n=s^{\prime}p^k$, $m=sp^k$, $p\nmid \gcd\left(s^{\prime},s\right)$, then the polynomials
   \begin{equation*}
     x^{s^{\prime}}+Ax^s+B \quad \mbox{and}\quad \dfrac{Ax^{sp^k}+B+\left(-Ax^s-B\right)^{p^k}}{p}
   \end{equation*}
   are coprime modulo $p$;
         \item \label{JKS:C5Main} if $p\nmid ABm$, then
     \[p^2\nmid \left(t^tB^{t-1}+(1-t)^{t-1}A^t\right).\]
   \end{enumerate}
\end{thm}

The analysis for the asymptotics in Section \ref{sec_asymp} pertaining to Theorems \ref{first_type} and \ref{second_type} depends on a corollary of the following classical result, due to Prachar, which gives an asymptotic for $S(X;r,m)$, the number of squarefree integers $\le X$ in the arithmetic progression $r \pmod{m}$.

\begin{thm}
\label{thm:Prachar}
{\rm \cite{Prachar1}}
Let $r$ and $m$ be positive integers with $\gcd(r,m)=1$. Then
\[
S(X;r,m) = \frac{X}{m \zeta(2)} \prod_{p \mid m} \left( 1 - \frac{1}{p^2} \right)^{-1} + O\left( X^{1/2} \right)
\]
where the implied constant is independent of all variables.
\end{thm}

We will eventually desire an estimate for
\[
S(X;r,m,q) := \left| \{ n \le X : \mu(n) \neq 0, n \equiv r (\text{mod } m), \gcd(n,q) = 1\} \right|
\]
when $\gcd(r,m) = \gcd(q,m)=1$. As an immediate consequence of Theorem \ref{thm:Prachar}, we obtain the following.

\begin{cor}
\label{Cor:Prachar}
Let $r$, $m$, and $q$ be positive integers with $\gcd(r,m) = \gcd(q,m)=1$. Then
\[
S(X;r,m,q) = X \left( \frac{\varphi(q)}{qm \zeta(2)} \right) \prod_{p \mid qm} \left( 1 - \frac{1}{p^2} \right)^{-1} + O \left(q X^{1/2} \right)
\]
where the implied constant is independent of all variables.
\end{cor}

For the proof of Theorem \ref{general_type}, we require certain asymptotic results related to the number of primes $p$ such that $F(p)$ is squarefree for a given polynomial $F(x)$.
The question of whether there exist infinitely many $n\in \Z$ such that $F(n)$ is squarefree for a given irreducible polynomial $F(x)\in \Z[x]$ has been investigated by numerous authors. Of course, certain restrictions must be imposed on $F(x)$ to avoid trivial situations. In particular, for every prime $q$, there should exist $z\in \Z$ such that $F(z)$ is not divisible by $q^2$. When $F(x)$ is linear, the fact that there exist infinitely many $n\in \Z$ such that $F(n)$ is squarefree is straightforward. Nagel \cite{Nagel} in 1922, and Estermann \cite{Estermann} in 1931, established the corresponding result for quadratics.
Erd\H{o}s \cite{Erdos53} settled the cubic case in 1953. No results are known unconditionally for the quartic case. Erd\H{o}s also asked if it is possible to find infinitely many primes $p$ such that $F(p)$ is squarefree. To answer this question, the linear case is again relatively easy, while the quadratic and cubic cases require significantly more effort \cite{Helfgott,Helfgott-cubic,Hooley,Pasten}.
The following asymptotic represents the current status of this situation.
\begin{thm}\label{Thm:Pasten}
   Let $F(x)\in \Z[x]$, and suppose that $F(x)$ factors into a product of distinct irreducibles, where the largest degree of any irreducible factor of $F(x)$ is $d$. Define 
   \[N_F\left(x\right)=\abs{\left\{p\le x : p \mbox{ is prime and } F(p) \mbox{ is squarefree}\right\}}\]
    Then, 
   \begin{equation}\label{Eq:AssPasten}
   N_F(x)\sim c_F\dfrac{x}{\log(x)},
   \end{equation}
   where
   \[c_F=\prod_{q  \mbox{ \rm {\tiny prime}}}\left(1-\dfrac{\rho_F\left(q^2\right)}{q(q-1)}\right)\]
   and $\rho_F\left(q^2\right)$ is the number of $z\in \left(\Z/q^2\Z\right)^{*}$ such that $F(z)\equiv 0 \pmod{q^2}$. The asymptotic \eqref{Eq:AssPasten} holds unconditionally if $d\le 3$, and holds conditionally on the abc-conjecture for number fields for $F(x)$ if $d\ge 4$.
 \end{thm}

 \begin{rem}
   The conditional part of Theorem \ref{Thm:Pasten} relies on the $abc$-conjecture for each number field $\Q(\alpha)$ where $\alpha$ varies over the irrational roots of $F(x)$. For a precise statement of this generalization of the classical $abc$-conjecture, see Conjecture 1.3 in \cite{Pasten}.
 \end{rem}

 Theorem \ref{Thm:Pasten}, which is well known among analytic number theorists, follows from work of Helfgott, Hooley and Pasten. To be more explicit, Hector Pasten has relayed to us (private communication) the following information. We consider first the unconditional part ($d\le 3$) of Theorem \ref{Thm:Pasten}. The case when $f(x)$ is a single irreducible cubic is settled in \cite[Main Theorem]{Helfgott-cubic}. Then one can use a key estimate from the proof of Helfgott's main theorem in \cite{Helfgott-cubic}, along with \cite[Lemma 3.2 and Lemma 3.3]{Pasten}  and an asymptotic  formula for $N_F(x)$ from \cite[p. 728]{Pasten}, to handle the cases when $f(x)$ is a product of only cubics, only quadratics or only linear polynomials. The situation when $f(x)$ is the product of irreducibles of various degrees, all smaller than 4, is addressed in \cite[Chapter 4]{Hooley-book}.
   The conditional part (when $d\ge 4$) of Theorem \ref{Thm:Pasten} follows from \cite[Theorem 1.1]{Pasten}.

The following corollary is immediate from Theorem \ref{Thm:Pasten}. 
\begin{cor}\label{Cor:Squarefree}
 Let $F(x)\in \Z[x]$, and suppose that $F(x)$ factors into a product of distinct irreducibles, where the largest degree of any irreducible factor of $F(x)$ is $d$. To avoid the situation when $c_F=0$, we suppose further that, for each prime $q$, there exists some $z\in \left(\Z/q^2\Z\right)^{*}$
  such that $F(z)\not \equiv 0 \pmod{q^2}$. If $d\le 3$, or if $d\ge 4$ and assuming the abc-conjecture for number fields for $F(x)$,
  there exist infinitely many primes $p$ such that $F(p)$ is squarefree.
\end{cor}
\begin{defn}\label{Def:Obstruction}
 If for some prime $q$, there does not exist $z\in \left(\Z/q^2\Z\right)^{*}$
  such that $F(z)\not \equiv 0 \pmod{q^2}$, we say that $F(x)$ has a \emph{local obstruction at $q$}.
\end{defn}

\section{More Preliminaries: New Results}

In this section, we develop some new machinery required for the proofs of Theorems \ref{first_type}, \ref{second_type} and \ref{general_type}. These results also include a new asymptotic used in the proofs of Theorems \ref{first_type} and \ref{second_type}.

\subsection{Identifying families of monogenic trinomials}
\label{sec_identify}

The following lemma will be used in the proofs of Theorems \ref{first_type} and \ref{general_type}. The general strategy of the lemma is to manufacture trinomials $f(x)=x^n+Ax^m+B$ for which statement \eqref{JKS:C4Main} in Theorem \ref{Thm:JKS1Main} is vacuously satisfied. This approach considerably reduces the work necessary to verify the statements in Theorem \ref{Thm:JKS1Main}.

\begin{lem}
\label{Lem:Len}
Let $n\ge 2$ be an integer, with $m\ge 1$ a proper divisor of $n$. Let $t=n/m$ and let $\kappa$ denote the squarefree kernel of $m$. Let $A$ and $B$ be positive integers with $\gcd(A,B)>1$, and define
\begin{equation}
\label{Eq:D}
D:=\dfrac{t^tB^{t-1}+(1-t)^{t-1}A^t}{\gcd(A,B)^{t-1}}.
\end{equation}
If $B$ and $D$ are squarefree, and $\gcd(A,B)\equiv 0\pmod{\kappa}$, then $f(x)=x^{n}+A x^{m}+B$ is monogenic. Moreover, $\Delta(f)$ is not squarefree if $m \ge 2$.
\end{lem}

\begin{proof}
Since $\gcd(A,B)>1$ and $B$ is squarefree, it follows that $f(x)$ is Eisenstein with respect to any prime divisor of $\gcd(A,B)$. Hence, $f(x)$ is irreducible. Since
\[
-(-1)^t(t-1)^{t-1}=(1-t)^{t-1} \quad \mbox{for any value of } t\ge 2,
\]
we have, from Theorem \ref{Thm:polydisc}, that
\[\Delta\left(f\right)=(-1)^{tm(tm-1)/2}m^{tm}B^{m-1}\left(\gcd(A,B)^{t-1}D\right)^m.
\]
We now verify that all the statements in Theorem \ref{Thm:JKS1Main} are true, which will establish that $f(x)$ is monogenic. Let $p$ be a prime divisor of $\Delta(f)$. Since $B$ is squarefree, we see easily that statement \eqref{JKS:C1Main} is true.

Next, suppose that $p\mid A$ and $p\nmid B$. Then $p\nmid \gcd(A,B)$, so that $p\nmid m$ since $\kappa\mid \gcd(A,B)$. Consequently, since $p\mid \Delta(f)$ and $p\mid A$, we deduce that $p\mid t$. Hence, since $t\ge 2$ and $p\nmid \gcd(A,B)$, it follows that $D \equiv 0 \pmod{p^2}$, contradicting the fact that $D$ is squarefree. We conclude that statement \eqref{JKS:C2Main} is vacuously satisfied.

Assume now that $p\nmid A$ and $p\mid B$. Therefore, as before, $p\nmid \gcd(A,B)$ and so $p\nmid m$. If $t=2$, then clearly $p\nmid (t-1)$. If $t>2$ and $p\mid (t-1)$, then $D\equiv 0 \pmod{p^2}$, which contradicts the fact that $D$ is squarefree. Thus, $\ell=0$ so that $a_1=0$, and therefore, $p\mid a_1$. Since $B$ is squarefree, $p\nmid b_2$, which implies that statement \eqref{JKS:C3Main} is true.

Statement \eqref{JKS:C4Main} is vacuously satisfied since there are no primes $p$ such that $p\nmid AB$ and $p\mid m$.

To check that statement \eqref{JKS:C5Main} is true, suppose that $p\nmid ABm$. Then, since $p\mid \Delta(f)$, we have that $p\mid \gcd(A,B)^{t-1}D$. Thus, if $\gcd(A,B)^{t-1}D\equiv 0 \pmod{p^2}$, then  $D \equiv 0 \pmod{p^2}$, since $p\nmid \gcd(A,B)$, once again contradicting the fact that $D$ is squarefree. Therefore, we conclude that $f(x)$ is monogenic by Theorem \ref{Thm:JKS1Main}.
\end{proof}

In Lemma \ref{Lem:Len}, each of the three hypotheses, $B$ is squarefree, $D$ is squarefree and $\gcd(A,B)\equiv 0 \pmod{\kappa}$, is necessary when $m\ge 2$, and the removal of any one of them results in numerous irreducible trinomials that are not monogenic. For purposes of illustration, we provide in Table \ref{Table:1} various possibilities of $f(x)$ and $K=\Q(\theta)$, where $f(\theta)=0$, when $n=4$ and $m=2$. We let $M$ denote monogenic and $NM$ denote non-monogenic. 

\renewcommand{\arraystretch}{1.25}
\begin{table}[h]
\begin{center}
\begin{tabular}{cccccc}
Hypothesis Removed & Example & $f(x)$ & $K$ & $\Delta(f)$ & $\Delta(K)$\\
\hline
$B$ is squarefree &  $x^4+2x^2+4$ & $NM$ & $NM$ & $2^{10}\cdot 3^2$ & $2^6\cdot 3^2$\\
$D$ is squarefree &  $x^4+2x^2+10$ & $NM$ & $M$ & $2^9\cdot 3^4\cdot5$ & $2^9\cdot 5$\\
$\gcd(A,B)\equiv 0 \pmod{\kappa}$ & $x^4+5x^2+5$& $NM$ & $NM$ & $2^4\cdot5^3$ & $5^3$\\
$\gcd(A,B)\equiv 0 \pmod{\kappa}$ & $x^4+7x^2+7$& $M$ & $M$ & $2^4\cdot 3^2\cdot 7^3$ & $2^4\cdot 3^2\cdot 7^3$\\
\end{tabular}
\end{center}
\caption{Candidates of Non-monogenic irreducible quartic trinomials}
\label{Table:1}
\end{table}

\vspace{-.18in}
\begin{rem}\label{Rem:Ref} 
   The discriminants $\Delta(f)$ of trinomials $f(x)$ in Table \ref{Table:1} are obtained by norms $\NN_k\left(\NN_{K/k}(f^{\prime}(\theta)\right)$, where $f^{\prime}(x)$ is the derivative of $f(x)$, and $k$ is the quadratic subfield of $K$. The first field discriminant $\Delta(K)$ and an integral basis $\O_K=\left\{1,\omega,\sqrt{\mu},\eta\right\}$, where
  \[\mu=-1+\sqrt{-3},\quad \omega=\frac{1+\sqrt{-3}}{2}, \quad \mbox{and} \quad\eta=\frac{(3-\sqrt{-3})\sqrt{\mu}}{2},\]
  are obtained by \cite[Theorem 1 and $C_4$ in Table $C^{\prime}$]{HSW}. The other three field discriminants are confirmed by \cite{HSW}. The monogeneity of the second quartic field is shown by the choice $\xi=i+\theta i$, where $i=\sqrt{-1}$ and $g(\xi)=0$, with $g(x)=(x^2-2)^2+(2x-3)^2$. The non-monogeneity of the first quartic field is proved by the evaluation modulo $3^3$ of the norm
  \[\NN_k\left(\NN_{K/k}\left((\xi-\xi^{\sigma})(\xi-\xi^{\sigma^3}\right)\right)\]
  for any monogenic candidate element $\xi=t\omega+u\sqrt{\mu}+v\eta$, under the condition
  \[\NN_k\left(\NN_{K/k}(\xi-\xi^{\sigma^2})\right)=2^6.\]
  Here, $\sigma$ denotes an embedding of $K$ into the algebraic closure of $K$ such that
  \[\theta^{\sigma}=-\sqrt{-1-\sqrt{-3}},\] and $\prod_{i=1}^3(\xi-\xi^{\sigma^i})$ is the different of $\xi\in K$. The non-monogeneity of the third quartic field is obtained by the modulo 2 evaluation of the representation matrix of any monogenic candidate $\xi=ti+u\phi_1+v\phi_2$, where the set $\{1,i,\phi_1,\phi_2\}$, with
  \[\phi_1=\frac{\theta^2-\theta-1}{2} \mbox{ and } \phi_2=\frac{\theta^3-1}{2},\] is an integral basis for $\O_K$ by \cite{HSW}. Namely, $[\O_K:\Z[\xi]]\equiv 0 \pmod{2}$.
\end{rem}

The following corollary identifies families of monogenic trinomials of the form $x^n +A x^m + A$ and follows from Lemma \ref{Lem:Len} via the special case $A = B \ge 2$.
\begin{cor}
\label{Cor:Len}
Let $n\ge 2$ be an integer, $m\ge 1$ be a proper divisor of $n$, and $\kappa$ denote the squarefree kernel of $m$.
Suppose that $f(x)=x^{n}+Ax^m+A$, where $A\ge 2$ is an integer such that $A \equiv 0 \pmod{\kappa}$. Let $t=n/m$ and let $D$ be as defined in \eqref{Eq:D} with $B=A$. Then $f(x)$ is monogenic if and only if $A$ and $D$  are squarefree. Moreover, $\Delta(f)$ is not squarefree if either $t\ge 2$ for $m\ge 2$, $t \ge 3$ for $m=1$, or $A$ is even.
\end{cor}
\begin{proof}
Note that here we have $D=t^t+(1-t)^{t-1}A$. If $A$ and $D$ are squarefree, then $f(x)$ is monogenic by Lemma \ref{Lem:Len}. Assume then that $f(x)$ is monogenic. By Theorem \ref{Thm:polydisc},
\[
\Delta(f)=(-1)^{tm(tm-1)/2}m^{tm}A^{tm-1}D^m.
\]
Since $f(x)$ is irreducible, all of the statements in Theorem \ref{Thm:JKS1Main} must be true. Let $p$ be a prime such that $p\mid A$. Then $p\mid \Delta(f)$ and we deduce from statement \eqref{JKS:C1Main} that $A$ is squarefree.

Now suppose that $p$ is a prime such that $p\mid D$. Then $p\mid \Delta(f)$. If $p\nmid A$, then $p^2\nmid D$ by statement \eqref{JKS:C5Main}. So suppose that $p\mid A$. Then $p\mid t$ and $p^2\mid t^t$ since $t\ge 2$. If $p^2\mid D$, then $p^2\mid (1-t)^{t-1}A$. But $p\nmid (1-t)$ since $p\mid t$. Thus, $p^2\mid A$, which contradicts the fact that $A$ is squarefree. Hence, $D$ is squarefree, which completes the proof.
\end{proof}

Corollary \ref{Cor:Len} is of interest, in part, because it shows that Lemma \ref{Lem:Len} captures all of the monogenic trinomials of the form $f(x)=x^n+Ax^m+A$, where $A\ge 2$ is an integer such that $A\equiv 0 \pmod{\kappa}$, and $m\ge 1$ is a proper divisor of $n\ge 2$.
\begin{rem}
  Corollary \ref{Cor:Len} is also true for $f(x)=x^n-Ax^m-A$,  where $A\ge 2$ is an integer such that $A \equiv 0 \pmod{\kappa}$.
\end{rem}

In the next proposition, which will be used to establish Theorem \ref{second_type}, we do not attempt to manufacture trinomials such that statement \eqref{JKS:C4Main} in Theorem \ref{Thm:JKS1Main} is vacuously satisfied, as was done in Lemma \ref{Lem:Len}. The price we pay is that we need to restrict the prime divisors of $m$ so that statement \eqref{JKS:C4Main} in Theorem \ref{Thm:JKS1Main} can be verified in a reasonable manner. Here, we restrict the prime divisors of $m$ to the set $\{2,3,5,7\}$ to make these computations tractable.

\begin{prop}
\label{Prop:Len2}
Let $n\ge 2$ be an integer, $m\ge 1$ be a proper divisor of $n$, and $\kappa$ denote the squarefree kernel of $m$. Suppose that $f(x)=x^{n}+Ax^m+A$, where $A\ge 2$ is an integer such that $A \equiv -1 \pmod{\kappa^2}$ where $\kappa \mid 210$. Let $t=n/m$ and $D$ be as defined in \eqref{Eq:D} with $B=A$. Then $f(x)$ is monogenic if and only if $A$ and $D$ are squarefree. Moreover, $\Delta(f)$ is not squarefree if either $t\ge 2$ for $m\ge 2$, $t \ge 3$ for $m=1$, or $A$ is even.
\end{prop}

\begin{proof}
Assume first that $A$ and $D$ are squarefree. Note that $f(x)$ is Eisenstein with respect to any prime divisor of $A$, which implies that $f(x)$ is irreducible. To establish that $f(x)$ is monogenic, we verify that all the statements in Theorem \ref{Thm:JKS1Main} are true. Statement \eqref{JKS:C1Main} is clearly true since $A$ is squarefree, while statements \eqref{JKS:C2Main} and \eqref{JKS:C3Main} are vacuously true.

To address statement \eqref{JKS:C5Main}, let $p$ be a prime divisor of
\[
\Delta(f)=(-1)^{tm(tm-1)/2}A^{tm-1}m^{tm}\left(t^t+A(1-t)^{t-1}\right)^m,
\]
such that $p\nmid Am$. Thus, $p\mid t^t+A(1-t)^{t-1}$. Then
\[
t^t A^{t-1} + (1-t)^{t-1} A^t=A^{t-1}\left(t^t+A(1-t)^{t-1}\right)\not \equiv 0 \pmod{p^2},
\]
since $t^t+A(1-t)^{t-1}$ is squarefree, and therefore statement \eqref{JKS:C5Main} is true.

Finally, we verify that statement \eqref{JKS:C4Main} is true. For any prime $p$ such that $p\nmid A$ and $p\mid m$, we write $m=sp^k$ and $s^{\prime}=tm/p^k=ts$, where $p\nmid \gcd(s,s^{\prime})$. We need to show that the polynomials
\[
G_1:=x^{ts}+Ax^s+A \quad \mbox{and} \quad G_2:=\dfrac{Ax^{sp^k}+A+\left(-Ax^s-A\right)^{p^k}}{p}
\]
are coprime modulo $p$. Expanding $G_2$ using the binomial theorem and rearranging, we get that
\[
G_2=\left(\frac{A+(-A)^{p^k}}{p}\right)\left(x^{sp^k}+1\right)+(-A)^{p^k}\frac{\sum_{j=1}^{p^k-1}\binom{p^k}{j}x^{s(p^k-j)}}{p}.
\]
Since $A\equiv -1 \pmod{p^2}$, a straightforward calculation yields
\[
\dfrac{A-A^{p^k}}{p}\equiv 0 \pmod{p} \quad  \mbox{and} \quad (-A)^{p^k}\equiv 1 \pmod{p}.
\]
Hence, we write
\begin{equation}
\label{Eq:G1}
\overline{G_1}=x^{ts}-x^s-1 \quad \mbox{and} \quad \overline{G_2}=\frac{\sum_{j=1}^{p^k-1}\binom{p^k}{j}x^{s(p^k-j)}}{p},
\end{equation}
where $\overline{*}$ is a convenient reduction since further arithmetic will take place modulo $p$. Of course, here we have that $p\in \{2,3,5,7\}$. A further analysis reveals, for $j\in \{1,2,\ldots ,p^k-1\}$, that
\[
\frac{\binom{2^k}{j}}{2}\equiv \left\{
    \begin{array}{cl}
      1 \pmod{2} & \mbox{if $j=2^{k-1}$}\\
      0 \pmod{2} & \mbox{otherwise,}\\
    \end{array}\right.\quad \frac{\binom{3^k}{j}}{3}\equiv \left\{
    \begin{array}{cl}
    1 \pmod{3} & \mbox{if $j\in \{3^{k-1},2\cdot 3^{k-1}\}$}\\
    0 \pmod{3} & \mbox{otherwise,}
    \end{array}\right.\]
    \[\frac{\binom{5^k}{j}}{5}\equiv \left\{
    \begin{array}{cl}
    1 \pmod{5} & \mbox{if $j\in \{5^{k-1},4\cdot 5^{k-1}\}$}\\
    2 \pmod{5} & \mbox{if $j\in \{2\cdot 5^{k-1},3\cdot 5^{k-1}\}$}\\
    0 \pmod{5} & \mbox{otherwise}\\
    \end{array}\right. \quad \mbox{and}\]
    \[\frac{\binom{7^k}{j}}{7}\equiv \left\{
    \begin{array}{cl}
    1 \pmod{7} & \mbox{if $j\in \{7^{k-1},6\cdot 7^{k-1}\}$}\\
    3 \pmod{7} & \mbox{if $j\in \{2\cdot 7^{k-1},5\cdot 7^{k-1}\}$}\\
    5 \pmod{7} & \mbox{if $j\in \{3\cdot 7^{k-1},4\cdot 7^{k-1}\}$}\\
    0 \pmod{7} & \mbox{otherwise.}\\
    \end{array}\right.
\]
Hence,
\begin{equation}
\label{Eq:G2}
\overline{G_2}=\left\{\begin{array}{cl}
      x^{s2^{k-1}} & \mbox{if $p=2$}\\
      x^{s3^{k-1}}\left(x^{s}+1\right)^{3^{k-1}} &  \mbox{if $p=3$}\\
      x^{s5^{k-1}}\left(x^s+1\right)^{5^{k-1}}\left(x^{2s}+x^s+1\right)^{5^{k-1}} &  \mbox{if $p=5$}\\
      x^{s7^{k-1}}\left(x^s+1\right)^{7^{k-1}}\left(x^{2s}+x^s+1\right)^{2\cdot 7^{k-1}} &  \mbox{if $p=7$.}
    \end{array}\right.
\end{equation}
Let $h(x)=\gcd\left(\overline{G_1},\overline{G_2}\right)$. Suppose that $h(x) \neq 1$ and let $h(\alpha)=0$ for $\alpha$ in an algebraic closure of $\Z / p\Z$. Thus, we see from \eqref{Eq:G2} that either $\alpha=0$, $\alpha^s+1=0$ or $\alpha^{2s}+\alpha^s+1=0$. But from \eqref{Eq:G1} we deduce that
\[
0=\overline{G_1}(\alpha)=\left\{\begin{array}{cl}
        -1 & \mbox{if $\alpha=0$}\\
        (-1)^t & \mbox{if $\alpha^s+1=0$,}
        \end{array} \right.
\]
which is impossible in any case. So, suppose that
\begin{equation}
\label{Eq:alpha}
\alpha^{2s}+\alpha^s+1=0,
\end{equation}
which implies that $\alpha^{3s}=1$. Write $t=3z+r$, where $r\in \{0,1,2\}$. Then
\begin{equation}
\label{Eq:alpha2}
0=\overline{G_1}(\alpha)=\alpha^{(3z+r)s}-\alpha^s-1=\alpha^{rs}-\alpha^s-1.
\end{equation}
If $r=0$, we see from \eqref{Eq:alpha2} that $\alpha=0$, which is impossible. If $r=1$, then, from \eqref{Eq:alpha2}, we arrive at the impossibility that $0=-1$. Finally, if $r=2$, then we have from \eqref{Eq:alpha2} that
\[
\alpha^{2s}-\alpha^s-1=0,
\]
which, combined with \eqref{Eq:alpha}, implies that $2\alpha^{2s}=0$. By \eqref{Eq:G2}, one may deduce $\alpha=0$, which by \eqref{Eq:alpha2} is again found to be impossible. It follows that $h(x)=1$, which proves that $f(x)$ is monogenic. The converse is straightforward and follows by an argument similar to the one at the end of the proof of Corollary \ref{Cor:Len}.
\end{proof}
\noindent
Computer evidence suggests that the factorization of $G_2$ modulo $p$ in Proposition \ref{Prop:Len2} becomes unwieldy for $p>7$, both with respect to $\kappa$, and the residue classes of the coefficients, preventing our reasonable attempts at a generalization.

Quartic trinomials of the form $x^4+Ax^2+B$ have previously been investigated in \cite{HSW,Kable} for integral and power bases. Note that when $B=A\equiv -1 \pmod{4}$, Proposition \ref{Prop:Len2} provides, in the more general setting of $\deg(f)=2^{k+1}$, the simple criterion that $f(x)=x^{2^{k+1}}+Ax^{2^k}+A$, where $k\ge 1$, $A\ge 3$ with $A\equiv -1 \pmod{4}$, is monogenic if and only if $A$ and $(A-4)$ are squarefree, or equivalently (since $\gcd(A,A-4)=1$), that $A(A-4)$ is squarefree.


For $n\ge 4$, the situation when $\kappa = 2$ and $A \equiv 1 \pmod{\kappa^2}$, which is handled by the same methods used in the proof of Proposition \ref{Prop:Len2}, is stated precisely in the following proposition. 

\begin{prop}
\label{kappa2}
Let $n \ge 4$ be an integer and $m = 2^k \ge 2$ be a proper divisor of $n$. Suppose that $f(x)=x^{n}+Ax^m+A$, where $A\ge 5$ is an integer such that $A \equiv 1 \pmod{4}$. Let $t=n/m$ and $D$ be as defined in \eqref{Eq:D}. Then $f(x)$ is monogenic if and only if $A$ and $D$ are squarefree and $t \not \equiv 2 \pmod{3}$. Moreover, $\Delta(f)$ is not squarefree.
\end{prop}
\begin{proof} For
 the trinomial \[f(x)=x^{t2^k}+Ax^{2^k}+A,\mbox{ where $k\ge 1$ and $t\ge 2$},\] we determine necessary and sufficient conditions so that each statement of Theorem \ref{Thm:JKS1Main} is true. Here we have that
\[\Delta(f)=(-1)^{t2^{k-1}(t2^k-1)}A^{t2^k-1}2^{tk2^k}D^{2^k},\] where $D=t^t+(1-t)^{t-1}A$.   Let $\P$ be the set of prime divisors of $\Delta(f)$.
Statement \eqref{JKS:C1Main} is true for all $p\in \P$ if and only if $A$ is squarefree. Note that, in this case, $f(x)$ is Eisenstein with respect to any prime divisor of $A$, which implies that $f(x)$ is irreducible.  Statements \eqref{JKS:C2Main} and \eqref{JKS:C3Main} are vacuously true for all $p\in \P$, while statement \eqref{JKS:C5Main} is true for all $p\in \P$ if and only if $D$ is squarefree. We turn finally to statement \eqref{JKS:C4Main}, where we see that \[p=2, \quad s^{\prime}=t \quad \mbox{and} \quad s=1.\] Then, the two polynomials appearing in statement \eqref{JKS:C4Main} are $G_1=x^t+Ax+A$ and
\begin{align*}
G_2&=\dfrac{Ax^{2^k}+A+(-Ax-A)^{2^k}}{2}\\
&=\left(\frac{A^{2^k}+A}{2}\right)x^{2^k}+A^{2^k}\sum_{j=1}^{2^k-1}\frac{{2^k\choose j}}{2}x^j+\frac{A^{2^k}+A}{2}.
\end{align*}
Note that ${2^k\choose j}\not \equiv 0 \pmod{4}$ if and only if $j=2^{k-1}$. Thus, since $A\equiv 1 \pmod{4}$, we have that
\[\overline{G_1}=x^t+x+1\quad \mbox{and}\quad \overline{G_2}=(x^2+x+1)^{2^{k-1}},\] where $\overline{G_i}$ is the reduction of $G_i$ modulo 2. Let $h(x)=\gcd(\overline{G_1},\overline{G_2})$, and suppose that $h(\alpha)=0$ for $\alpha$ in some algebraic closure of $\Z / 2\Z$. Then, since  $\overline{G_2}(\alpha)=0$, we see that $\alpha^3-1=(\alpha-1)(\alpha^2+\alpha+1)=0$, so that the order of $\alpha$ is 3. Hence,
\begin{align*}
  \gcd(\overline{G_1},\overline{G_2})\ne 1 & \Longleftrightarrow G_1(\alpha)=0\\
  & \Longleftrightarrow \alpha^t+\alpha+1=0\\
  & \Longleftrightarrow \alpha^t+\alpha^2=0\\
  & \Longleftrightarrow \alpha^{t-2}+1=0\\
  & \Longleftrightarrow t\equiv 2 \pmod{3},
\end{align*}
which completes the proof.
\end{proof}
We then have the following immediate corollary of Proposition \ref{kappa2}.
\begin{cor}\label{Cor:kappa2}
 Let $k$ and  $A$ be integers such that $k\ge 1$ and $A\ge 5$, with $A\equiv 1 \pmod{4}$. Then $f(x)=x^{2^{k+1}}+Ax^{2^k}+A$ is not monogenic.
\end{cor}
\begin{proof}
  Since $t=2$, it follows from Proposition \ref{kappa2} that $f(x)$ is not monogenic.
\end{proof}

\subsection{Asymptotic for Theorems \ref{first_type} and \ref{second_type}} \label{sec_asymp}

The next result is an asymptotic that allows us to count, for a fixed integer $n\ge 2$ and a proper divisor $m\ge 1$ of $n$, the number of monogenic trinomials of the form $f(x)=x^n+Ax^m+A$ with $A \le X$, under certain restrictions on $A$. Specifically, this asymptotic will be used in the proofs of Theorem \ref{first_type} and Theorem \ref{second_type}.

Consider an integer $\beta$ and positive integers $\rho$, $\gamma$, $\alpha$, $\alpha_0$, and $\beta_0$ that satisfy the following.
\begin{equation}
\label{Eq:DanRestrictions}
\boxed{
\begin{array}{rl}
1. & \gcd(\alpha_0 \beta_0  \rho,\gamma)=1=\gcd(\alpha,\beta)\\
2. & \mbox{for each prime $p \mid \beta$, we also have $p^2 \mid \beta$}\\
3. & \alpha_0 \mbox{ is a squarefree divisor of $\alpha$}\\
4. & \beta_0 \mbox{ is a squarefree divisor of $\beta$}\\
5. & \alpha\beta_0 \rho + \beta \not \equiv 0 \Mod{p^2} \mbox{ for every $p \mid \gamma$}\\
\end{array}
}
\end{equation}

We wish to count the number of positive, squarefree integers $y \le X$ which are congruent to $\rho$ modulo $\gamma^2$ with $\gcd(y,\alpha_0 \beta_0)=1$ such that $\alpha \beta_0 y + \beta$ is also squarefree. Precisely, we define the function
\begin{multline}\label{Eq:U}
U(X;\rho,\gamma,\alpha, \alpha_0, \beta, \beta_0):=\\
 \left| \left\{  y \le X : y \equiv \rho \, (\text{mod } \gamma^2), \, \gcd(y,\alpha_0 \beta_0)=1,\, \mu(y) \neq 0, \, \mu(\alpha \beta_0 y+\beta) \neq 0 \right\} \right|.
\end{multline}
Since the context will always be clear, we simply refer to the function in \eqref{Eq:U} as $U(X)$, where the dependence on the other variables is implicit.

\begin{thm} \label{mainasymp} Given the restrictions on the variables $\rho$, $\gamma$, $\alpha$, $\beta$, $\alpha_0$ and $\beta_0$ described in \eqref{Eq:DanRestrictions}, we have
\[
U(X) = X \left( \frac{ \varphi(\alpha_0 \beta_0)}{\alpha_0 \beta_0 \gamma^2 \zeta(2)} \right) \prod_{p \mid \alpha_0 \beta_0 \gamma} \left( 1 - \frac{1}{p^2} \right)^{-1}  \prod_{p \nmid \alpha \beta \gamma} \left( 1 - \frac{1}{p^2-1} \right) + O\left( X^{3/4} \right),
\]
where the implied constant is dependent on $\gamma$, $\alpha$, $\alpha_0$, $\beta$, and $\beta_0$.
\end{thm}

\begin{proof}
For brevity, we let $F(X) = \alpha \beta_0 X + \beta$. A standard method for detecting squarefree integers will be employed here. By the well-known identity $\sum_{d \mid n} \mu(d) = \delta(n,1)$, we have
\begin{align}
U(X) &= \sum_{\substack{y \le X \\ \mu(y) \neq 0 \\ y \equiv \rho \Mod{\gamma^2} \\
\gcd(y,\alpha_0 \beta_0)=1}} \sum_{d^2 \mid \alpha \beta_0 y + \beta} \mu(d) = \sum_{\substack{y \le X \\ d \le \sqrt{F(y)}}} \mu(d) \sum_{e \le F(y)/d^2} \chi\left(d^2e\right) \label{swaporder} 
\end{align}
where
\[
\chi(n) :=
    \left\{
     \begin{array}{ll}
       1 & \text{if } \frac{n-\beta}{\alpha \beta_0} \equiv \rho \Mod{\gamma^2} \text{ is a squarefree positive integer co-prime to } \alpha_0 \beta_0; \\
       0 & \text{otherwise}.
     \end{array}
   \right.
\]
We first set out to understand exactly when $y = \frac{d^2 e - \beta}{\alpha \beta_0}$ is a positive integer. Therefore we examine the congruence
\begin{align}
d^2 e - \beta \equiv 0 \Mod{\alpha \beta_0}, \label{Dancong}
\end{align}
which has solutions in $e$ only when $\gcd(d,\alpha)=1$ since $\gcd(\alpha,\beta)=1$. Furthermore, the condition $\gcd(y,\beta_0)=1$, coupled with our goal to detect squarefree values of $y$, means we may restrict to the situation where $\gcd(d,\beta)=1$ by the hypotheses on $\beta$ and $\beta_0$. Under these considerations, (\ref{Dancong}) has a unique solution in $e$. It follows that the positive integer solutions to (\ref{Dancong}) are of the form $e = e_0 + s \alpha \beta_0$ for some minimal $e_0 > 0$ and $s \ge 0$. For these positive integers, we observe
\[
y =\frac{d^2 (e_0 + s \alpha \beta_0) - \beta}{\alpha \beta_0} = d^2 s + \frac{d^2 e_0 - \beta}{\alpha \beta_0} \equiv - \alpha^{-1} (\beta / \beta_0) \Mod{d^2}.
\]
These observations, together with (\ref{swaporder}), give
\begin{align}
U(X) &= \sum_{\substack{d \le \sqrt{F(X)} \\ \gcd(d,\alpha \beta)=1}} \mu(d) \xi(d) S \left(X;R_d,\lcm\left(\gamma^2,d^2 \right), \alpha_0 \beta_0 \right) + O\left( \sqrt{X} \right) \nonumber\\
& = \sum_{\substack{d \le X_0 \\ \gcd(d,\alpha \beta)=1}} \mu(d) \xi(d) S \left(X;R_d,\lcm\left(\gamma^2,d^2 \right), \alpha_0 \beta_0 \right) + O\left( \frac{X}{X_0} \right) \label{beforeasymp}
\end{align}
where $X_0 \le \sqrt{F(X)}$ is dependent on $X$ and chosen later, $\xi(d)$ is the indicator function for whether or not the congruence classes
\begin{equation}
- \alpha^{-1} (\beta / \beta_0) \Mod{d^2} \quad \text{and} \quad \rho \Mod{\gamma^2} \label{congs}
\end{equation}
have non-empty intersection, and $R_d$ is the residue modulo $\lcm(\gamma^2,d^2)$ that results upon use of the Chinese remainder theorem on the congruences in ($\ref{congs}$) if $\xi(d)=1$ where $R_d$ is arbitrary otherwise.

Now we may insert the estimate from Corollary \ref{Cor:Prachar} into (\ref{beforeasymp}). Let $U_M(X)$ denote the quantity that arises from the insertion of the  main term of Corollary \ref{Cor:Prachar} into (\ref{beforeasymp}), and let $U_E(X) := U(X) - U_M(X)$ represent the rest. We bound $U_E(X)$ by
\begin{align}
\sum_{\substack{d \le X_0 \\ \gcd(d,\alpha \beta)=1}} \mu(d) \xi(d) O \left( X^{1/2} \right) + O\left( \frac{X}{X_0} \right) = O \left( X_0 X^{1/2} + \frac{X}{X_0} \right) \label{errorbound}
\end{align}
so that only analyzing $U_M(X)$ remains. Note that

\begin{align}
U_M(X) &= \sum_{\substack{d \le X_0 \\ \gcd(d,\alpha \beta)=1}} \mu(d) \xi(d) \frac{ \varphi(\alpha_0 \beta_0) X}{(\alpha_0 \beta_0) \lcm\left(\gamma^2,d^2 \right) \zeta(2)} \prod_{p \mid \alpha_0 \beta_0 \lcm(\gamma^2,d^2)} \left( 1 - \frac{1}{p^2} \right)^{-1} \nonumber \\
\begin{split}
&= \frac{\varphi(\alpha_0 \beta_0) X}{\alpha_0 \beta_0 \gamma^2 \zeta(2)} \prod_{p \mid \alpha_0 \beta_0 \gamma} \left( 1 - \frac{1}{p^2} \right)^{-1}\\ & \qquad \qquad \qquad \times \sum_{\substack{d \le X_0 \\ \gcd(d,\alpha \beta)=1}} \frac{\mu(d)}{d^2} \xi(d) \gcd\left(\gamma^2,d^2 \right) \prod_{\substack{p \nmid \gamma \\ p \mid d}} \left( 1 - \frac{1}{p^2} \right)^{-1}. \label{mult}
\end{split}
\end{align}

We have arranged the expression in (\ref{mult}) so that the summands are multiplicative with respect to $d$. This enables us to complete the sum over all $\gcd(d,\alpha \beta)=1$ and factor it into an Euler product, up to some manageable error. By (\ref{mult}) and bounding the tail of the completed sum trivially,
\begin{multline*}
U_M(X) = \frac{\varphi(\alpha_0 \beta_0) X}{\alpha_0 \beta_0 \gamma^2 \zeta(2)} \prod_{p \mid \alpha_0 \beta_0 \gamma} \left( 1 - \frac{1}{p^2} \right)^{-1} \prod_{\substack{p \mid \gamma \\ p \nmid \alpha \beta}} \left( 1 - \xi(p) \right)\\
 \times \prod_{\substack{p \nmid \gamma \\ p \nmid \alpha \beta}} \left( 1 + \frac{\xi(p)}{p^2} \left( 1 - \frac{1}{p^2} \right)^{-1} \right) + O \left( \frac{X}{X_0} \right)
 \end{multline*}
where one clearly obtains $\xi(p) = 1$ for $p \nmid \gamma$. In consideration of $\xi(p)$ for the case $p \mid \gamma$, we may consider the subcases $p \nmid \beta_0$ and $p \mid \beta_0$. For the former, the hypothesis $\alpha \beta_0 \rho + \beta \not \equiv 0 \, (\text{mod } p^2)$ prevents the congruences in (\ref{congs}) from having non-empty intersection, so $\xi(p)= 0$. For the latter, $\xi(p)=1$ would imply $p \mid \rho$ by the hypothesis on $\beta$, contradicting $\gcd(\rho, \gamma)=1$. Thus,
\begin{equation}
U_M(X) = \frac{\varphi(\alpha_0 \beta_0) X}{\alpha_0 \beta_0 \gamma^2 \zeta(2)} \prod_{p \mid \alpha_0 \beta_0 \gamma} \left( 1 - \frac{1}{p^2} \right)^{-1}  \prod_{p \nmid \alpha \beta \gamma} \left( 1 - \frac{1}{p^2-1} \right) + O \left( \frac{X}{X_0} \right). \label{mainterm}
\end{equation}
In consideration of $U(X) = U_M(X) + U_E(X)$ and lines (\ref{errorbound}) and (\ref{mainterm}) with the choice $X_0 = X^{1/4}$, the proof of Theorem \ref{mainasymp} is complete.
\end{proof}

\section{Proofs of Main Results}
\label{main_results}

We now apply the asymptotic results developed in Section \ref{sec_asymp} to the families of monogenic trinomials identified in Section \ref{sec_identify} to prove Theorems \ref{first_type}, \ref{second_type} and \ref{general_type}. We end with a proof of Proposition \ref{galois_prop}, showing these families can have relatively small Galois groups.

\subsection{The proof of Theorem \ref{first_type}}
\begin{proof}
Since $A\equiv  0\pmod{\kappa}$ and $A$ is squarefree, we can write $A=a\kappa$, where $\gcd(a,\kappa)=1$. We let $t=n/m$ and
\[
\begin{array}{lclcl}
    \gamma=\rho=1& \qquad \quad & \beta_0=\gcd(t,\kappa) & \qquad \quad & y=a\\
    \beta=(-1)^{t-1}t^t & & \alpha_0=\kappa/\beta_0 && \alpha=(t-1)^{t-1}\alpha_0.\\
\end{array}
\]
It is easy to check that these variables satisfy conditions \eqref{Eq:DanRestrictions}. We let
\[
\widehat{D}:=\alpha\beta_0y+\beta=(t-1)^{t-1}\kappa a+(-1)^{t-1}t^t=(t-1)^{t-1}A+(-1)^{t-1}t^t,
\]
so that $\widehat{D}=(-1)^{t-1}D$. We apply Theorem \ref{mainasymp} to deduce the growth rate of the number of squarefree positive integers $A\equiv 0 \pmod{\kappa}$, such that $D$ is also squarefree. By Corollary \ref{Cor:Len}, the proof of the theorem is complete.
\end{proof}

In Table \ref{Table:2}, we provide some examples of the actual count of the number of monogenic trinomials $f(x)=x^n+Ax^m+A$, with $1 \le A \le X=10000$ satisfying the hypotheses of Corollary \ref{Cor:Len}, versus the main term of \eqref{first_asymp} rounded to the nearest integer. The actual counts in Table \ref{Table:2} were determined using Maple 9.5.
\begin{table}[h]
\begin{center}
\begin{tabular}{ccccc}
$n$ & $m$ & Actual count & Main term \\ \hline
24 & 12 & 460 & 461\\
19 & 1 & 5549 & 5548\\
14 & 7 & 624 & 618 \\
12 & 3  & 1380 & 1383 \\
8 & 4  & 1617 & 1614\\
\end{tabular}
\end{center}
\caption{Number of monogenic trinomials versus main term of \eqref{first_asymp}}
\label{Table:2}
\end{table}


\subsection{The proof of Theorem \ref{second_type}}
\begin{proof}
We let $t=n/m$ and
\[\begin{array}{lclcl}
    \gamma=\kappa& \qquad \quad & \alpha_0=\beta_0=1 & \qquad  \quad & y=A\\
    \beta=(-1)^{t-1}t^t & &  \rho=\kappa^2-1 && \alpha=(t-1)^{t-1}.\\
\end{array}
\]
It is easy to check that these variables satisfy the first four conditions in \eqref{Eq:DanRestrictions}. To verify that the last condition in \eqref{Eq:DanRestrictions} is satisfied, we must show that
\[
C(t):=(t-1)^{t-1}\left(\kappa^2-1\right)+(-1)^{t-1}t^t\not \equiv 0 \pmod{p^2},
\]
for every prime $p$ dividing $\kappa$. Since it is easy to see that
\[
C\left(t+p^2(p-1)\right)\equiv C(t) \pmod{p^2},
\]
we conclude that the sequence $C(t)$ is periodic modulo $p^2$ with period at most $p^2(p-1)$. Note that here we have that $p\in \{ 2,3,5,7\}$. It is then a simple calculation to establish, for each of these primes $p$, that $C(t)\not \equiv 0 \pmod{p^2}$ for any integer $t\ge 2$. Hence, all conditions in   \eqref{Eq:DanRestrictions} are satisfied. We let
\[
\widehat{D}:=\alpha\beta_0y+\beta=(t-1)^{t-1}A+(-1)^{t-1}t^t,
\]
so that $\widehat{D}=(-1)^{t-1}D$. We apply Theorem \ref{mainasymp} to deduce the growth rate of the number of squarefree positive integers $A\equiv -1 \pmod{\kappa^2}$, such that $D$ is also squarefree. By Proposition \ref{Prop:Len2}, the proof of the theorem is complete.
\end{proof}

In Table \ref{Table:3}, we provide some examples of the actual count of monogenic trinomials $f(x)=x^n+Ax^m+A$, with $1 \le A \le X=10000$ satisfying the hypotheses of Proposition \ref{Prop:Len2}, versus the main term of \eqref{second_asymp} rounded to the nearest integer. The actual counts in Table \ref{Table:3} were determined using Maple 9.5.
\begin{table}[ht]
\begin{center}
\begin{tabular}{ccccc}
$n$ & $m$ & Actual count & Main term \\ \hline
24 & 12 & 232 & 231 \\
14 & 7 & 102 & 103\\
12 & 3  & 688 & 691\\
8 & 4  & 1619 & 1614\\
\end{tabular}
\end{center}
\caption{Number of monogenic trinomials versus main term of \eqref{second_asymp}}
\label{Table:3}
\end{table}\\


\subsection{The proof of Theorem \ref{general_type}}
\begin{proof}
Let $a=A/r\kappa$ and consider the polynomial
\[
F(x)=t^tx^{t-1}+(1-t)^{t-1}a^tr\kappa\in \Z[x].
\]
We wish to apply Corollary \ref{Cor:Squarefree} to $F(x)$. To ensure we are not in a trivial situation, note first that $F(x)$ has no repeated zeros. We also need to verify that $F(x)$ has no local obstructions. That is, for each prime $q$, we must show that there exists some $z\in \left(\Z/q^2\Z\right)^{*}$
  such that $F(z)\not \equiv 0 \pmod{q^2}$.

  Suppose first that $t\equiv 0 \pmod{q}$. Then, either
  \begin{equation}\label{Eq:Either}
  (1-t)^{t-1}a^tr\kappa\not \equiv 0 \pmod{q} \qquad \mbox{or} \qquad (1-t)^{t-1}a^tr\kappa\equiv 0 \pmod{q}.
  \end{equation} Clearly, in the first case of \eqref{Eq:Either}, we have that $(1-t)^{t-1}a^tr\kappa\not \equiv 0 \pmod{q^2}$. In the second case of \eqref{Eq:Either}, we must have that $r\kappa\equiv 0 \pmod{q}$ since $\gcd((1-t)a,t)=1$. But $r\kappa\not \equiv 0 \pmod{q^2}$ since $r\kappa$ is squarefree. Hence, in the second case of \eqref{Eq:Either} as well, we have that $(1-t)^{t-1}a^tr\kappa\not \equiv 0 \pmod{q^2}$. Therefore, since $t\ge 2$, it follows for either case of \eqref{Eq:Either} that
  \[F(z)\equiv (1-t)^{t-1}a^tr\kappa\not \equiv 0 \pmod{q^2} \quad \mbox{for any $z\in \left(\Z/q^2\Z\right)^{*}$.}\]

  Suppose now that $t\not \equiv 0 \pmod{q}$. We again have the two possibilities of \eqref{Eq:Either}. In the second case of \eqref{Eq:Either}, if
  $(1-t)^{t-1}a^tr\kappa \equiv 0 \pmod{q^2}$, then we easily observe that
  \[F(1)\equiv t^t\not\equiv 0 \pmod{q^2}.\] On the other hand, if $(1-t)^{t-1}a^tr\kappa \not \equiv 0 \pmod{q^2}$, then we also claim that
   \[F(1)=t^t+(1-t)^{t-1}a^tr\kappa\not\equiv 0 \pmod{q^2}.\]
   To see this, assume to the contrary that $F(1)\equiv 0 \pmod{q^2}$. Then
   \[qt^t\equiv qt^t+q(1-t)^{t-1}a^tr\kappa\equiv qF(1)\equiv 0 \pmod{q^2},\]
   which contradicts the fact that $t\not \equiv 0 \pmod{q}$.

   Finally, we turn to the first possibility in \eqref{Eq:Either}, still under the assumption that $t\not \equiv 0 \pmod{q}$. In this situation, we see that $t-1\not \equiv 0\pmod{q}$. Let $z\in \left(\Z/q^2\Z\right)^{*}$. If $F(z)\not \equiv 0 \pmod{q^2}$, then we are done. So, suppose that $F(z)\equiv 0 \pmod{q^2}$. Let $\rho\in \left(\Z/q^2\Z\right)^{*}$ be a primitive root modulo $q^2$. If $F(z\rho)\equiv 0 \pmod{q^2}$, then
   \[F(z\rho)-F(z)\equiv t^t(z\rho)^{t-1}-t^tz^{t-1}\equiv t^tz^{t-1}\left(\rho^{t-1}-1\right)\equiv 0 \pmod{q^2},\]
   from which it follows that
   \begin{equation}\label{Eq:PR}
   \rho^{t-1}-1\equiv 0 \pmod{q^2}.
   \end{equation} However, \eqref{Eq:PR} contradicts the fact that $\rho$ is a primitive root modulo $q^2$ since $\ord_{q^2}(\rho)=q(q-1)$ and $t-1\not \equiv 0 \pmod{q}$. Hence, $F(z\rho)\not \equiv 0 \pmod{q^2}$ and $F(x)$ has no local obstructions.

Thus, by Corollary \ref{Cor:Squarefree}, there exist infinitely many primes $p$ such that $F(p)$ is squarefree, where the existence of the primes $p$ is unconditional when $2\le t\le 4$, and conditional on the $abc$-conjecture for number fields when $t\ge 5$. If, for any of these primes $p>A$, we let $B=pr\kappa$, then clearly $B$ is squarefree, and
\begin{align*}
\dfrac{t^tB^{t-1}+(1-t)^{t-1}A^t}{\gcd(A,B)^{t-1}}&=\dfrac{t^tp^{t-1}r^{t-1}\kappa^{t-1}+(1-t)^{t-1}a^tr^t\kappa^t}{r^{t-1}\kappa^{t-1}}\\
&=t^tp^{t-1}+(1-t)^{t-1}a^tr\kappa\\
&=F(p)
\end{align*}
is also squarefree. Hence, by Lemma \ref{Lem:Len}, $f(x)$ is monogenic, completing the proof of Theorem \ref{general_type}.
\end{proof}

\subsection{The proof of Proposition \ref{galois_prop}}

\begin{proof}
Suppose that $f(\theta)=0$, and let $K=\Q(\theta)$.
Then, factoring $f(x)$ over $K$, we get
\[
f(x)=g(x) \left(x^{m}-\theta^{m}\right)\\
 =g(x) \prod_{d\mid m} \theta^{\phi(d)}\Phi_{d}\left(\dfrac{x}{\theta}\right) \quad \text{where} \quad g(x)=A+\sum_{j=0}^{t-1}\theta^{jm}x^{(t-j-1)m}
\]
and $\Phi_d(x)$ denotes the $d$th cyclotomic polynomial.
Thus, the $m$ zeros of $f(x)/g(x)$ are precisely
\[
\theta,\, \theta\zeta,\, \theta\zeta^{2},\,\ldots ,\,\theta\zeta^{m-1},
\]
where $\zeta$ is a primitive $m$th root of unity, and $\left[\Q(\theta,\zeta):\Q\right]=\phi(m)tm$. If $g(x)$ is irreducible over $\Q(\theta,\zeta)$, and $g(\eta)=0$, then
\[
\left[\Q(\eta,\theta,\zeta):\Q(\theta,\zeta)\right]=(t-1)m.
\] Continuing, we see that no more roots of unity are required to split $f(x)$ completely and, by induction, we conclude that the ``worst-case" scenario is
\[
[L:\Q]=\phi(m)\prod_{j=0}^{t-1}(t-j)m=\phi(m) m^t t!,
\]
where $L$ is a splitting field over $\Q$ for $f(x)$.
\end{proof}

\section*{Acknowledgments}
The authors thank the referee for the helpful suggestions, including an enhancement to Table \ref{Table:1}, and Remark \ref{Rem:Ref}.

\end{document}